\documentclass[12pt]{amsart}
\usepackage{amsmath,amsthm,amsfonts,amssymb,eucal}

\renewcommand {\a}{ \alpha }
\renewcommand{\b}{\beta}

\newcommand{\g}{\gamma}
\newcommand{\G}{\Gamma}

\newcommand{\vark}{\varkappa}
\renewcommand{\d}{\delta}

\renewcommand{\L}{\Lambda}
\newcommand{\z}{\zeta}

\newcommand{\p}{\partial}

\newcommand{\Om}{\Omega}

\newcommand{\oq}{\ {\raise 7pt\hbox{${\scriptstyle\circ}$}}
\kern -7pt{
\hbox{$Q$}}}

\newcommand{\R}{ \mathbb R}

\newcommand {\GS}{\mathfrak S}

\newcommand {\BI}{\mathbf I}

\newcommand {\BS}{\mathbf S}

\newcommand {\BT}{\mathbf T}

\newcommand {\bx}{\mathbf x}

\newcommand {\bk}{\mathbf k}

\newcommand {\bl}{\mathbf l}
\newcommand {\bw}{\mathbf w}
\newcommand {\bz}{\mathbf z}
\newcommand {\by}{\mathbf y}

\newcommand {\bs}{\mathbf s}
\newcommand {\bu}{\mathbf u}
\newcommand {\bn}{\mathbf n}
\newcommand {\bj}{\mathbf j}

\newcommand{\bzero}{\mathbf 0}

\newcommand{\SN}{{\sf{N}}}

\newcommand {\bmu}{\boldsymbol\mu}

\newcommand {\boldeta}{\boldsymbol\eta}

\newcommand {\bxi}{\boldsymbol\xi}

\newcommand{\lu}{\langle}
\newcommand{\ru}{\rangle}

\newcommand{\CP}{\mathcal P}

\newcommand{\CC}{\mathcal C}

\newcommand{\plainC}[1]{\textup{{\textsf{C}}}^{#1}}

\newcommand{\plainL}[1]{\textup{{\textsf{L}}}^{#1}}
\newcommand{\plainl}[1]{\textup{{\textsf{l}}}^{#1}}

\newcommand{\1}
{{\,\vrule depth3pt height9pt}{\vrule depth3pt height9pt}
{\vrule depth3pt height9pt}{\vrule depth3pt height9pt}\,}

\DeclareMathOperator {\dist} {{dist}}

\DeclareMathOperator{\op}{{Op}}

\DeclareMathOperator{\supp}{{supp}}

\newtheorem{thm}{Theorem}[section]
\newtheorem{cor}[thm]{Corollary}
\newtheorem{lem}[thm]{Lemma}
\newtheorem{prop}[thm]{Proposition}

\theoremstyle{definition}
\newtheorem{defn}[thm]{Definition}

\newtheorem{rem}[thm]{Remark}

\numberwithin{equation}{section}

\newcommand{\bee}{\begin{equation}}
\newcommand{\ene}{\end{equation}}
\newcommand{\bees}{\begin{equation*}}
\newcommand{\enes}{\end{equation*}}
\newcommand{\bes}{\begin{split}}
\newcommand{\ens}{\end{split}}

\newcommand{\bet}{\begin{thm}}
\newcommand{\ent}{\end{thm}}
\newcommand{\bel}{\begin{lem}}
\newcommand{\enl}{\end{lem}}
\newcommand{\bec}{\begin{cor}}
\newcommand{\enc}{\end{cor}}
\newcommand{\bep}{\begin{proof}}
\newcommand{\enp}{\end{proof}}
\newcommand{\ber}{\begin{rem}}
\newcommand{\enr}{\end{rem}}

\newcommand{\Z}{\mathbb Z}

\begin{document}
\hoffset -4pc

\title
[Schatten-von Neumann properties ]
{{On the Schatten-von Neumann  properties of some pseudo-differential operators}}
\author[A.V. Sobolev]{Alexander V. Sobolev}
 \address{Department of Mathematics\\ 
 University College London\\
Gower Street\\ London\\ WC1E 6BT UK}
\email{asobolev@math.ucl.ac.uk}
\keywords{ Pseudo-differential operators, Schatten-von Neumann classes}
\subjclass[2010]{Primary 47G30; Secondary 35S05, 47B10, 47B35}

\begin{abstract}
 We obtain a number of explicit estimates for quasi-norms of 
 pseudo-differential operators in the Schatten-von Neumann classes 
 $\GS_q$ with $0<q\le 1$. 
The estimates are applied to derive semi-classical bounds for operators 
with smooth or non-smooth symbols.    
\end{abstract}

\maketitle

\section{Introduction}

When working with compact pseudo-differential operators 
it is often important to know how fast their singular values (or eigenvalues) decay. 
These properties 
are  conveniently stated in terms of the classical Schatten-von Neumann 
classes  $\GS_p, p >0$, or even more general ideals  
$\GS_{p, q}$, $p, q >0$. 
We refer to \cite{BS_U}, 
\cite{BS}, \cite{GK} and \cite{Simon} for information on compact operator ideals.

Not surprisingly, the Schatten-von Neumann properties 
of pseudo-differential operators 
are determined by smoothness of their symbols. 
The first bound in the trace class $\GS_1$ was obtained  
in  \cite{ShT}, and later reproduced in \cite{Sh}, Proposition 27.3, 
and \cite{Robert}, Theorem II-49, see also \cite{H0}. 
Some useful $\GS_1$-bounds were obtained in the 
much more recent paper \cite{Roz}. 
 The other ideals 
$\GS_p$, $\GS_{p, q}$ were studied e.g. in 
\cite{Arsu}, \cite{BuNi}, \cite{HRT},  \cite{Rond}, 
\cite{Toft}, and there one can find further references. 
The fundamental paper \cite{BS_U}
contains $\GS_{p, q}$-estimates 
for integral operators in terms of
smoothness of their kernels.

In spite of a relatively large number of available results, they 
are not always practically useful since in applications 
one often  needs more detailed information. In this paper 
we obtain some explicit bounds for 
Schatten-von Neumann norms of 
various pseudo-differential operators aiming at 
applications in semi-classical analysis. 
Let  $p = p(\bx, \by, \bxi)$, $\bx, \by, \bxi\in\R^d$, $d\ge 1$, 
be a smooth \textit{amplitude}. 
For any $\a >0$ introduce the standard notation 
for the pseudo-differential operator with amplitude $p$:
\begin{equation}\label{ampl1:eq}
(\op^{\rm a}_\a (p)) u(\bx) = \biggl(\frac{ \a}{2\pi}\biggr)^d
\iint e^{i\a(\bx-\by)\bxi} p(\bx, \by, \bxi) u(\by) d\by d\bxi,
\end{equation}
for any Schwartz class function $u$. In the literature one uses more often 
the reciprocal value $\a^{-1}$ which is 
interpreted as the Planck's constant. 
It is natural for us to study a somewhat more general variant of the
 operator \eqref{ampl1:eq}. 
Let $\BT = \{t_{jk}\}$ be a non-degenerate $(2\times 2)$-matrix with real-valued 
entries. We concentrate on the operators 
\begin{equation}\label{opt:eq}
\begin{cases}
\op^{\rm a}_\a(p_\BT),\ p_{\BT}(\bx, \by, \bxi) 
= p( \bw, \bz, \bxi),\\[0.2cm] 
\textup{with}\ 
\  \bw = t_{11}\bx +t_{12}\by, 
\bz = t_{21}\bx + t_{22} \by.
\end{cases}
\end{equation}
This choice of the amplitude allows us to derive bounds for 
various standard quantizations of pseudo-differential operators. For a smooth 
\textit{symbol}  $a = a(\bx, \bxi)$ and a number $t\in [0, 1]$ 
we define the $t$-quantization as the pseudo-differential operator
\begin{equation}\label{tquant:eq}
\bigl(\op_{\a, t} (a) u\bigr)(\bx) 
= \biggl(\frac{\a}{2\pi}\biggr)^d \iint e^{i\a(\bx-\by)\cdot\bxi}
a\bigl((1-t)\bx+ t\by, \bxi\bigr) u(\by) d\by d\bxi,
\end{equation}
for any Schwartz class function $u$, see e.g. \cite{Robert}, Ch. 2, \S 4. 
It is clear that this operator can be written as 
\begin{equation}\label{tquant1:eq}
\op_{\a, t}(a) = \op^{\rm a}_{\a}(p_\BT),\ \ \textup{with} \ \ 
p(\bw, \bz, \bxi) = a(\bw, \bxi),\ \ 
\BT = 
\begin{pmatrix}
1-t & t\\
-1 & 1
\end{pmatrix}.  
\end{equation}
In this formula the choice of the second row in the matrix $\BT$ is 
unimportant as long as $\BT$ remains non-degenerate. 
Note also that formally 
$\bigl(\op_{\a, t}(a)\bigr)^* = \op_{\a, 1-t}(\overline a)$. 
The values $t = 0$ and $t=1$ give the standard 
``left" and ``right" quantizations. 
In these cases the operator \eqref{tquant:eq} 
has the symbol  $a(\bx, \bxi)$ (for $t=0$) 
or $a(\by, \bxi)$ (for $t = 1$).  
In the literature one sometimes uses for 
them the notation $\op^{\rm l}_\a(a)$ 
and $\op^{\rm r}_\a(a)$ respectively. 
Another important example is the Weyl quantization:
\begin{equation*}
\op^{\rm W}_\a(a) = \op_{\a, \frac{1}{2}}(a),
\end{equation*} 
which has the advantage that $\bx$ and $\by$ enter the 
definition \eqref{tquant:eq} symmetrically. If the symbol $a$ depends only on $\bxi$ 
then the operators \eqref{tquant:eq} for different values of $t$ 
coincide with each other and we write simply $\op_\a(a)$.

If the functions $p$ and $a$ above are sufficiently smooth and decay 
sufficiently fast at infinity then the operators \eqref{opt:eq},\eqref{tquant:eq} 
belong to $\GS_q$ with a suitable $q>0$.  
The aim of the paper is to study this property
for $q\in (0, 1]$.   
Our results are divided in three groups. 
First in Section \ref{general:sect} 
we obtain general estimates in $\GS_q$ for $\a = 1$, 
see Theorems \ref{full:thm}, \ref{symbol:thm}. 
The $\GS_q$-bounds for the operators \eqref{opt:eq} seem to be 
quite useful from the practical point of view. 
In particular 
they allow us to study the operators of the form 
$h_1 \op_{1, t} (a) h_2$, 
$t\in [0,1]$ with the weights $h_1, h_2$ whose supports are disjoint, 
and to control explicitly the dependence on the distance between the supports, 
see Theorem  \ref{symbol:thm}(2).  
Our approach stems from a simple idea suggested in the paper \cite{Roz} where 
trace class properties of pseudo-differential operators were studied. 
In fact, our results can be viewed 
as quantitative variants of 
Proposition 3.2 and Theorem 3.5. 
from \cite{Roz}, extended to the ideals $\GS_q$, $q \le 1$. 
As the classes 
$\GS_q$ with $q < 1$ are not normed, the obtained  
$\GS_q$-estimates for the operators \eqref{opt:eq} and \eqref{tquant:eq}
involve the so-called \textit{lattice quasi-norms}(see \eqref{brackh:eq}) 
for the amplitudes/symbols and their derivatives (for $q = 1$ 
these quasi-norms are simply $\plainL1$-integral norms). 
The estimates in $\GS_q$ with $q > 1$  
are also of great interest, but they are likely to be 
stated in different terms, cf. \cite{Arsu}, \cite{BuNi}, \cite{Toft}, 
and thus they are not discussed here. 

Sections \ref{semi_class:sect} and \ref{nonsmooth:sect} are devoted to 
applications. 
In Section \ref{semi_class:sect} we use Theorems \ref{full:thm} and \ref{symbol:thm} 
to derive 
estimates for large values of the parameter 
$\a$, which can be interpreted as the semi-classical regime. 
These results are stated in terms of the scaling properties 
of the symbols which makes them flexible and convenient for applications. 
Section \ref{nonsmooth:sect} is concerned with semi-classical 
bounds for operators with discontinuous symbols. 
The discontinuities are introduced as characteristic functions $\chi_\L(\bx)$ 
and $\chi_\Om(\bxi)$ of some Lipschitz 
domains $\L$ and $\Om$. We derive $\GS_q$-semi-classical estimates
for the Hankel-type operators $\chi_\L \op_{\a, t}(a)(I-\chi_\L)$ and 
$\chi_\L P_{\Om, \a}(I-\chi_{\L})$, $P_{\Om, \a} = \op_\a(\chi_\Om)$. 
This study is motivated by the trace asymptotics 
for Wiener-Hopf and Hankel operators 
with discontinuous symbols, both 
classical, see e.g.\cite{Land_Wid}, 
\cite{Widom1},  and multi-dimensional, see \cite{Sob}, \cite{Sob_LMS}.  

A number of 
estimates similar to the ones in Sections \ref{semi_class:sect} and 
\ref{nonsmooth:sect} have been established in \cite{Sob} 
for the trace class $\GS_1$. However some applications 
in Mathematical Physics, and in particular in Quantum Information Theory, 
call for estimates 
in the classes of compact operators with a faster decay of the singular values, 
see \cite{GiKl}, \cite{HLS}. This was the main incentive for the current paper.

To conclude the Introduction we make some notational conventions. 
Throughout the paper we denote by $C$ or $c$ with or without indices 
various positive constant whose value is unimportant. 
The notation $B(\bu, r)$ is used for the open ball in $\R^d, d\ge 1,$ 
of radius $r>0$ centred at the point $\bu\in\R^d$. The characteristic function 
of the ball $B(\bu, r)$ is denoted by $\chi_{\bu, r}$.

\textbf{Acknowledgements.} The author is grateful 
to H. Leschke and W. Spitzer for introducing him to problems 
in Quantum Information Theory involving pseudo-differential operators 
with discontinuous symbols, and for useful remarks on the paper. 
This work was supported by EPSRC grant EP/J016829/1.

\section{General estimates in $\GS_q$-ideals with $q\in (0, 1]$: smooth symbols}
\label{general:sect}

\subsection{Ideals $\GS_q$} 
The notation $\GS_q, q>0,$
is standard for the set of all 
compact operators $A$  on a Hilbert space with singular values $s_k(A)$, 
$k=1, 2, \dots$,  
for which the functional 
\begin{equation*}
\|A\|_{\GS_q} = \biggl(\sum_{k=1}^\infty s_k(A)^q\biggr)^{\frac{1}{q}}
\end{equation*}
is finite. For $q\ge 1$ this functional defines a natural norm on $\GS_q$, 
whereas  for $q <1$ it defines a quasi-norm. Nevertheless one has the triangle 
inequality of the form 
\begin{equation}\label{triangle:eq}
\|A_1+A_2\|_{\GS_q}^q \le \|A_1\|_{\GS_q}^q + \|A_2\|_{\GS_q}^q, \ 0<q \le 1,
\end{equation}
see \cite{Rot} and \cite{BS}, p.262, and the following 
H\"older-type inequality:
\begin{equation}\label{holder:eq}
\|A_1 A_2\|_{\GS_q}\le \|A_1\|_{\GS_{q_1}}\|A_2\|_{\GS_{q_2}},\ 
q^{-1} = q_1^{-1} + q_2^{-1}, 0<q_1, q_2 \le \infty,
\end{equation}
see \cite{BS}, p. 262. 

A crucial technical point in the study of the operators \eqref{opt:eq} 
 is to estimate suitable $\GS_q$-(quasi)- norms for the operators 
 $h \op_1(a)$, $h = ~h(\bx)$, $a = ~a(\bxi)$,
which have been studied quite extensively. We need the following
estimate  which is a 
slight generalization of the bound found 
in \cite{BS_U}, Theorem 11.1 (see also \cite{BKS}, Section 5.8),
and quoted  in \cite{Simon}, Theorem 4.5 for $s\in [1, 2]$.

Let $\CC_\bu\subset \R^m$ be a cube centred at $\bu\in \R^m$ 
with the edge of unit length. 
For a function $h\in\plainL{r}_{\textup{\tiny loc}}(\R^m)$, $r \in (0, \infty),$ 
denote
\begin{equation}\label{brackh:eq}
\begin{cases}
\1 h\1_{r,\delta} =
\biggr[\sum_{\bn\in\Z^d}
\biggl(\int_{\CC_{\bn}} |h(\bx)|^r d\bx\biggr)^{\frac{\d}{r}}\biggr]^{\frac{1}{\d}},\ \
0<\d<\infty,\\[0.3cm]
\1 h\1_{r, \infty} =
\sup_{\bu\in\R^d}
\biggl(\int_{\CC_{\bu}} |h(\bx)|^r d\bx\biggr)^{\frac{1}{r}},\ \
\d = \infty.
\end{cases}
\end{equation}
These functionals are sometimes called \textit{lattice  
quasi-norms} (norms for $r, \d\ge 1$).  
If $\1 h\1_{r, \d}<\infty$ we say that  
$h\in \plainl{\d}(\plainL{r})(\R^m)$.

\begin{prop}\label{BS:prop}
Suppose that $f\in  \plainl{q}(\plainL{2})(\R^n)$ 
and
$g\in  \plainl{q}(\plainL{2})(\R^m)$, 
with some $q \in (0, 2]$. 
Let $K: \plainL2(\R^m) \to \plainL2(\R^n)$ be the operator with the kernel 
\[
f(\bx) e^{i\bx\cdot \BS \by  }g(\by),\ \bx\in\R^n, \by\in \R^m,
\]
where $\BS: \R^m\to\R^n$ is a linear map. 
Then  
\[
\| K\|_{\GS_q} \le C_q \1 f\1_{2, q} \1 g\1_{2, q},
\]
with a constant $C_q = C_q(\BS)$ depending 
only on the number $s_0$ in the bound  
$\max_{jk}|s_{jk}|\le~s_0$ 
for the entries $s_{jk}, j=1, 2, \dots, n$; $k = 1, 2, \dots m,$ of the matrix $\BS$.
\end{prop}
 
We do not give the proof as it repeats that of \cite{BS_U}, Theorem 11.1 
almost word to word. 

\subsection{Estimates for the operators \eqref{opt:eq}} 
Now we need to specify the conditions on the matrix $\BT = \{t_{jk}\}, j,k=1,2$. 
The end results require $\BT$ to be non-degenerate, i.e. 
$\BT \in  GL(2, \R)$. 
For convenience we sometimes assume that 
\begin{equation}\label{T:eq}
t_{11}+t_{12} = 1,
\end{equation}
and denote
\begin{equation}\label{tau:eq}
\tau = t_{21}+t_{22}.
\end{equation}
Using the inverse of $\BT$, we can recover $\bx$ and $\by$ from the 
vectors $\bw$ and $\bz$ defined in \eqref{opt:eq}: 
\begin{equation}\label{xy:eq}
\begin{cases}
(\det\BT) \bx = t_{22}\bw - t_{12}\bz, \ 
(\det\BT)\by = -t_{21} \bw + t_{11}\bz,\\[0.2cm]
\textup{so}\ \  
(\det\BT)(\bx- \by) = \tau \bw - \bz.
\end{cases}
\end{equation}
We assume that 
\begin{equation}\label{t0:eq}
\max_{jk} |t_{jk}|\le t_0,\ \ |\det\BT|\ge \d_0,
\end{equation}
with some fixed positive numbers $t_0$, $\d_0$. 
In the estimates below 
the constants may be dependent on $t_0$ and $\d_0$. We provide 
appropriate comments in every instance. 

Assuming that $p(\ \cdot\ ,\ \cdot\ , \bxi )\in\plainL1(\R^{2d})$,  
introduce the ``double" Fourier transform:
\begin{equation*}
\hat p(\boldeta, \bmu, \bxi)
= \frac{1}{(2\pi)^{ d }}
\iint e^{-i\bw\cdot \boldeta-i\bz\cdot\bmu} p(\bw, \bz, \bxi)
d\bw d\bz.
\end{equation*}

\begin{lem}\label{q_amplitude1:lem} 
Let $\BT$ be an arbitrary $(2\times 2)$-matrix 
with real-valued entries. 
Suppose that
$p(\ \cdot\ , \ \cdot\ , \bxi)\in \plainL1(\R^{2d})$ for a.e. $\bxi\in\R^d$. 
Let $h_1, h_2~\in~\plainl{2q}(\plainL{2})(\R^d)$, and let  
$\hat p\in\plainl{q}(\plainL{1})(\R^{3d})$ with some $q\in (0, 1]$. 
 Then the operator
$h_1 \op^{\rm a}_1(p_\BT) h_2$ belongs to $\GS_q$ and  
\begin{equation}\label{trace_amplitude1:eq}
\| h_1 \op^{\rm a}_1(p_\BT) h_2\|_{\GS_q}\le   
C_q \1 h_1\1_{2, 2q} \1 h_2\1_{2, 2q} \1 \hat p\1_{1, q},
\end{equation}
with a constant $C_q = C_q(t_0)$. 
\end{lem}

\begin{proof}
Represent the amplitude $a$ via its Fourier transform 
\begin{equation*}
p(\bw, \bz, \bxi)
= \frac{1}{(2\pi)^{d}}
\iint e^{i\bz\cdot \boldeta + i\bw\cdot \bmu} \hat p(\boldeta, \bmu, \bxi)
d\boldeta d\bmu,
\end{equation*}
and rewrite $A = h_1 \op_1^{\rm a}(p_\BT) h_2$ as follows: 
\begin{equation*}
A = B_1 B_2^*,
\end{equation*}
where $B_j: \plainL2(\R^{3d})\to \plainL2(\R^d), \ j = 1, 2$, are the operators 
with the kernels
\begin{align*}
b_1(\bx; \boldeta, \bmu, \bxi)
= &\ \frac{1}{(2\pi)^d}
h_1(\bx) e^{i\bx\cdot(\bxi + t_{11}\boldeta + t_{21} \bmu)} 
\hat p(\boldeta, \bmu, \bxi)^{\frac{1}{2}},\\[0.2cm]
b_2(\bx; \boldeta, \bmu, \bxi)
= &\ \frac{1}{(2\pi)^d} h_2(\bx)
e^{i\bx\cdot (\bxi- t_{12}\boldeta - t_{22}\bmu)} 
|\hat p(\boldeta, \bmu, \bxi)|^{\frac{1}{2}},
\end{align*}
where $z^{1/2} = z |z|^{-1/2}$ for any $z\not =0$.
By Proposition \ref{BS:prop}, 
\begin{equation*}
\| B_j\|_{\GS_{2q}}\le C_q(t_0) \1 |\hat p|^{1/2}\1_{2, 2q} \1 h_j\1_{2, 2q}, j=1, 2.
\end{equation*}
Now \eqref{trace_amplitude1:eq} follows from \eqref{holder:eq}.
\end{proof}

It is usually more convenient to write $\GS_q$-estimates in terms of the
amplitudes themselves, and not their Fourier transforms. 
For $m, n = 0, 1, \dots,$ let 
\begin{align}
P_{n, m}(\bw, \bz, \bxi; p) = &\ 
\frac{1}{1+|\bz-\tau\bw|^m}
 \sum_{n_1, n_2=0}^{n}\sum_{l=0}^m
|\nabla_{\bw}^{n_1} \nabla_{\bz}^{n_2}\nabla_{\bxi}^l
 p(\bw, \bz, \bxi)|,\label{Pmn:eq}\\[0.3cm]
Q_{n, m}(\bxi; p) = &\ 
\iint P_{n, m}(\bw, \bz, \bxi) 
d\bw d\bz.\label{Qmn:eq}
\end{align}
The parameter $\tau$ is defined in \eqref{tau:eq}. 

\begin{cor} \label{symbol:cor} 
Let the matrix $\BT$ and the functions 
$h_1, h_2$ be as in Lemma \ref{q_amplitude1:lem},\ and let 
$Q_{n, m}(p)\in \plainl{q}(\plainL{1})(\R^d)$ 
with some $q\in (0, 1]$, and 
\begin{equation}\label{n:eq}
n = [dq^{-1}] + 1.
\end{equation}
Then
\begin{equation}\label{symbol0:eq}
\|h_1 \op^{\rm a}_1(p_\BT) h_2\|_{\GS_q}\le C_q\1 h_1 \1_{2, 2q}
\1 h_2\1_{2, 2q} 
\1 Q_{n, 0}(p)\1_{1, q},
\end{equation}
with a constant $C_q = C_q(t_0)$. 
\end{cor}

\begin{proof} 
Integrating by parts, we get:
\begin{equation*}
|\hat p(\boldeta, \bmu, \bxi)|
\le C(n) (1+|\boldeta|)^{-n} (1+ |\bmu|)^{-n}
\sum_{n_1, n_2 = 0}^{n}
 \iint
|\nabla_{\bw}^{n_1} \nabla_{\bz}^{n_2} p(\bw, \bz, \bxi)|
d\bw d\bz. 
\end{equation*}
For $n = [dq^{-1}] + 1$ the function on the right-hand side 
belongs to $\plainl{q}(\plainL{1})(\R^{3d})$, 
and its quasi-norm  \eqref{brackh:eq} does not exceed 
$ C\1 Q_{n, 0}(p)\1_{1, q}$. Now \eqref{symbol0:eq} 
follows from Lemma \ref{q_amplitude1:lem}. 
\end{proof}
 
\begin{lem} \label{amplitude:lem} 
Suppose that $\BT\in GL(2, \R)$ satisfies \eqref{T:eq}. 
Let $h_1, h_2$ be as in Lemma \ref{q_amplitude1:lem},\ and let 
$Q_{n, m}\in \plainl{q}(\plainL{1})(\R^d)$ 
with some $q\in (0, 1]$, with $n$ satisfying \eqref{n:eq}, and some 
$m = 0, 1, \dots$. Then  
\begin{equation}\label{amplitude:eq}
\|h_1 \op^{\rm a}_1(p_\BT) h_2\|_{\GS_q}
\le  
C_q
\1 h_1\1_{2, 2q} 
\1 h_2\1_{2,2q}\1 Q_{n, m}(p)\1_{1, q},
\end{equation}
with a constant $C_q=C_q(t_0)$.  
\end{lem}

\begin{proof} 
Let 
\begin{equation*}
\CP^{(\pm)}_\bx = \bigl(1\pm i(\det\BT)^2\bx\cdot\nabla_{\bxi}\bigr)
\bigl(1+(\det\BT)^2|\bx|^{2}\bigr)^{-1}.
\end{equation*}
Clearly,
$\CP^{(-)}_\bx e^{i\bxi\cdot\bx} = e^{i\bxi\cdot\bx}$, so 
integrating by parts $m$ times,
we get the following formula for the kernel of the operator 
$\op^{\rm a}_1(p_\BT)$:
\begin{equation*}
\frac{1}{(2\pi)^d} \int e^{i\bxi\cdot(\bx - \by)} 
p^{(m)}_\BT(\bx, \by, \bxi) d\bxi,
\end{equation*}
with 
\begin{equation*}
p^{(m)}_\BT(\bx, \by, \bxi) = 
\bigl(\CP^{(+)}_{\bx-\by}\bigr)^m p_\BT(\bx, \by, \bxi),
\end{equation*}
so by \eqref{xy:eq} 
\begin{equation*}
p^{(m)}(\bw, \bz, \bxi) = 
\bigl(1+|\tau\bw-\bz|^{2}\bigr)^{-m}
\bigl(1+ i(\det\BT)(\tau\bw-\bz)\cdot\nabla_{\bxi}\bigr)^m
p(\bw, \bz, \bxi).
\end{equation*}
Now it is straightforward to see that
\begin{equation*}
P_{n, 0}(\bw, \bz, \bxi; p^{(m)})
\le C(t_0) P_{n, m}(\bw, \bz, \bxi; p).
\end{equation*}
By Corollary \ref{symbol:cor} this implies the proclaimed result.
\end{proof}

In the next Theorem we replace the $(2,2q)$-quasi-norms of functions $h_1, h_2$
by much weaker ones. 
 
\begin{thm}\label{full:thm} 
Suppose that $\BT\in GL(2, \R)$ satisfies \eqref{T:eq}. 
Let $h_1, h_2~\in~\plainl{\infty}(\plainL{2})(\R^d)$, 
and let $P_{n, m}\in \plainl{q}(\plainL{1})(\R^{3d})$ 
with some $q\in (0, 1]$, with $n$ satisfying \eqref{n:eq} and some 
$m = 0, 1, \dots$. Then  
\begin{equation}\label{amplitude1:eq}
\|h_1 \op^{\rm a}_1(p_\BT) h_2\|_{\GS_q} 
\le  C_{q, m} \1 h_1\1_{2, \infty}
\1 h_2\1_{2, \infty} \1 P_{n, m}(p)\1_{1, q},
\end{equation}
with a constant $C_{q, m} = C_{q, m}(t_0, \d_0)$ depending on $t_0$ and $\d_0$. 
\end{thm}

\begin{proof}  
Let us define a convenient 
partition of unity. 
The open balls $B(\bj, 2\sqrt d), \bj\in\Z^d,$ form 
a covering of $\R^d$. Let $\{\psi_\bj\}$ be an associated partition of unity 
such that 
\begin{equation}\label{uniform:eq}
|\nabla_\bx^k\psi_{\bj}(\bx)|\le C_k,\ k = 0, 1, \dots, 
\end{equation}
uniformly in $\bj\in\Z^d$. 
Let us consider the operator $\op_1^{\rm a}(p^{(\bj, \bs)}_\BT)$ with the amplitude 
\[
p^{(\bj, \bs)}(\bw, \bz, \bxi) = \psi_\bj(\bw) \psi_\bs(\bz) p(\bw, \bz, \bxi). 
\]
Since $\bw \in B(\bj, 2\sqrt d), \bz\in B(\bs, 2\sqrt d)$, we have 
\begin{gather*}
\bx \in B(\bl, R),\ \by\in B(\bn, R),\ \ 
\textup{with}\ \ 
R = \frac{4t_0\sqrt d}{\d_0},\\[0.2cm] 
\bl = \frac{t_{22}\bj - t_{21} \bs}{\det \BT},\ 
\bn = \frac{-t_{21}\bj + t_{11} \bs}{\det\BT},
\end{gather*}
see \eqref{xy:eq}. 
 Consequently
 \begin{equation*}
  h_1 \op_1^{\rm a}(p^{(\bj, \bs)}_\BT) h_2
 =  h_1 \chi_{\bl, R}\op_1^{\rm a}
 (p^{(\bj, \bs)}_\BT) h_2\chi_{\bn, R}, 
 \end{equation*}
 and hence by Lemma \ref{amplitude:lem},
\begin{align*}
 \| h_1 \op_1^{\rm a}(p^{(\bj, \bs)}_\BT) h_2\|_{\GS_q}
\le C_q\1 h_1 \chi_{\bl, R}\1_{2, 2q}
\1 h_2\chi_{\bn, R}\1_{2, 2q} 
\1 Q_{n, m}(p^{(\bj, \bs)})\1_{1, q}.
\end{align*}
The first two factors are estimated 
by $C\1 h_1\1_{2, \infty}$
and $C\1 h_2\1_{2, \infty}$ respectively, 
with some constant $C=C(t_0, \d_0)$. 
Thus by the triangle inequality 
\eqref{triangle:eq}
\begin{align*}
\| h_1 \op^{\rm a}_1(p_\BT) h_2\|_{\GS_q}^q\le &\ 
\sum_{\bj, \bs}  \| h_1 \op_1^{\rm a}(p^{(\bj, \bs)}_\BT) h_2\|_{\GS_q}^q\\[0.3cm]
\le &\ C_q\1 h_1 \1_{2, \infty}^q
\1 h_2\1_{2, \infty}^q \sum_{\bj, \bs}
\1 Q_{n, m}(p^{(\bj, \bs)})\1_{1, q}^q.
\end{align*}
Remembering that the number of intersecting balls $B(\bj, 2\sqrt d)$ is 
uniformly bounded, we can estimate the sum on the right-hand side 
by $\tilde C\1 P_{n, m}(p)\1_{1, q}^q$. This completes the proof. 
\end{proof}

\subsection{Estimates for the operators \eqref{tquant:eq}} 
Theorem \ref{full:thm} allows amplitudes 
independent of $\bz$, e.g. it allows one to consider 
$t$-pseudo-differential operators \eqref{tquant:eq}. 
We isolate this observation in a separate theorem. 
For a symbol $a = a(\bx, \bxi)$ denote
\begin{equation}\label{F:eq}
\begin{split}
F^{\circ}_{n, m}(\bw, \bxi; a)
= &\ \sum_{k =0}^n 
|\nabla_{\bw}^k \nabla_{\bxi}^m a(\bw, \bxi)|,\\
\ F_{n, m}(\bw, \bxi; a) 
= &\ \sum_{l=0}^m F^{\circ}_{n, l}(\bw, \bxi; a),\ n, m = 0, 1, \dots.
\end{split}
\end{equation}
The constants in the next theorem are independent 
of $t\in [0, 1]$.

\begin{thm}\label{symbol:thm} Let 
$h_1, h_2~\in~\plainl{\infty}(\plainL{2})(\R^d)$, 
let $n$ be as in \eqref{n:eq}, and $q\in (0, 1]$ be some number. 
 \begin{enumerate}
 \item 
 Suppose that $F_{n, n}(a)\in \plainl{q}(\plainL{1})(\R^{2d})$. 
Then for any $t\in [0, 1]$ we have  
\begin{equation}\label{symbol:eq}
\|h_1 \op_{1, t}(a) h_2\|_{\GS_q} 
\le  C_q \1 h_1\1_{2, \infty}
\1 h_2\1_{2, \infty}
\1 F_{n, n}(a)\1_{1, q}.
\end{equation}
\item
Suppose that 
the distance between the supports of the functions $h_1, h_2$ is 
at least $r\ge 1$. If 
$F^{\circ}_{n, m}(a) \in \plainl{q}(\plainL{1})(\R^{2d})$, 
$m \ge n$, then for any $t\in [0, 1]$ 
we have 
\begin{equation}\label{symbol_sep:eq}
\|h_1 \op_{1, t}(a) h_2\|_{\GS_q} 
\le  C_{q, m} r^{\frac{d}{q} - m}\1 h_1\1_{2, \infty}
\1 h_2\1_{2, \infty}
\1 F^{\circ}_{n, m}(a)\1_{1, q}.
\end{equation}
\end{enumerate}
\end{thm}

\begin{proof}  
Use Theorem \ref{full:thm} with $p(\bw, \bz, \bxi) = a(\bw, \bxi)$ 
and the matrix 
\begin{equation}\label{newt:eq}
\BT = 
\begin{pmatrix}
1-t & t\\
-1 & 1
\end{pmatrix},
\end{equation}
so that $\tau = 0$, see \eqref{tau:eq}.  
By definitions \eqref{Pmn:eq} and \eqref{F:eq},
\begin{equation*}
P_{n, m}(\bw, \bz, \bxi; p)
\le \frac{F_{n, m}(\bw, \bxi; a)}{1+|\bz|^m}.
\end{equation*}
To estimate $\1 P_{n, m}(p)\1_{1, q}$ write for any $\bk, \bs, \bj\in\Z^d$:
\begin{equation*}
\int\limits_{\CC_\bk}
\int\limits_{\CC_\bs}
\int\limits_{\CC_\bj}
P_{n, m}(\bw, \bz, \bxi; p) 
d\bw d\bz d\bxi
\le C \frac{1}
{1 + |\bs|^m} 
\int\limits_{\CC_\bk}
\int\limits_{\CC_\bj}
F_{n, m}(\bw, \bxi; a) 
d\bw d\bxi. 
\end{equation*}
Consequently, for $m\ge n$,
\begin{equation*}
\1 P_{n, m}(p)\1_{1, q}^q
\le C\1 F_{n, m}(a) \1_{1, q}^q \sum_{\bs\in\Z^d} \frac{1}{1 + |\bs|^{mq}}
\le C' \1 F_{n, m}(a) \1_{1, q}^q.
\end{equation*}
Now Theorem \ref{full:thm} with $m = n$ implies \eqref{symbol:eq}. 

Proof of \eqref{symbol_sep:eq}. 
Let $\z\in\plainC\infty(\R)$ be a function 
such that  
\begin{equation}\label{zeta:eq}
\z(u) = 
\begin{cases}
1,\   |u| \ge 1;\\ 
0,\   |u|\le \frac{1}{2}.   
 \end{cases}
 \end{equation}
Note that 
\begin{equation*}
h_1 \op_{1, t}(a) h_2
= h_1 \op^{\rm a}_1 (g_\BT) h_2,\ 
g(\bw, \bz, \bxi) = \z(|\bz|r^{-1})
a(\bw, \bxi),
\end{equation*}
where the matrix $\BT$ is defined as in 
\eqref{newt:eq}. 
We use Theorem \ref{full:thm} again but in a slightly different way than above -- 
first we implement integration by parts similar to the one 
done in the proof of Lemma \ref{amplitude:lem}.  
Let $\CP^{(\pm)}_\bz =(\pm i\bz\cdot\nabla_{\bxi})|\bz|^{-2}$. Clearly,
$\CP^{(+)}_\bz e^{-i\bxi\cdot\bz} = e^{-i\bxi\cdot\bz}$, so,
integrating by parts $m$ times,
we get the following formula for the kernel of the operator $\op^{\rm a}_1(g_\BT)$:
\begin{equation*}
\frac{1}{(2\pi)^d} \int e^{i\bxi\cdot(\bx - \by)} g^{(m)}_\BT(\bx, \by, \bxi) d\bxi,
\end{equation*}
with
\begin{equation*}
g^{(m)}(\bw, \bz, \bxi) = 
\bigl(\CP^{(-)}_{\bz}\bigr)^m g(\bw, \bz, \bxi).
\end{equation*}
It is straightforward to see that
\begin{equation*}
P_{n, 0}(\bw, \bz, \bxi; g^{(m)})
\le C \frac{F^{\circ}_{n, m}(\bw, \bxi; a)}{r^m + |\bz|^m},
\end{equation*}
with a constant independent of $r$. 
Arguing as in the first part of the proof 
we get the bound
\begin{equation*}
\1 P_{n, 0}(g^{(m)})\1_{1, q}^q
\le C\1 F^{\circ}_{n, m}(a) \1_{1, q}^q 
\sum_{\bs\in\Z^d} \frac{1}{r^{mq} + |\bs|^{mq}}
\le C' \1 F^{\circ}_{n, m}(a) \1_{1, q}^q \ \ r^{d-mq}.
\end{equation*}
Theorem \ref{full:thm} with $m\ge n$ leads to  
\eqref{symbol_sep:eq}. 
\end{proof}

As the next Theorem shows, in the case $d = 1$, 
when $h_1$ and $h_2$ have disjoint supports, one can sometimes 
allow symbols $a$ depending only on $\xi$. Here and below 
we use $x$ and $\xi$ for one-dimensional variables. 

\begin{thm}\label{onedim_smooth:thm}
Let 
$h_1, h_2~\in~\plainl{\infty}(\plainL{2})(\R)$ be two functions such that 
\begin{equation*}
h_1(x) = 0,\ \ \textup{a.e.}\  x > -\frac{r}{2},\quad  h_2(x) = 0,
\ \ \textup{a.e.}\ x < \frac{r}{2}, 
\end{equation*}
with some $r\ge 1$. Let $q\in (0, 1]$ be some number, and let 
$h = [q^{-1}]+1$.  
Suppose that $a = a(\xi)$ satisfies the condition 
$\p^m a\in\plainl{q}(\plainL{1})(\R)$, for some $m\ge n$. 
Then we have 
\begin{equation}\label{symbol_sep_onedim:eq}
\|h_1 \op_1(a) h_2\|_{\GS_q} 
\le  C_{q, m} r^{\frac{1}{q} - m}\1 h_1\1_{2, \infty}
\1 h_2\1_{2, \infty}
\1 \p^m a \1_{1, q}.
\end{equation}
\end{thm}

\begin{proof}
As in the proof of the previous theorem,
\begin{equation*}
h_1 \op_1(a) h_2 = h_1 \op_1^{\rm a} (g) h_2,\ g(x, y, \xi) = \z(|x-y|r^{-1})a(\xi),
\end{equation*}
where $\z$ is a sdefined in \eqref{zeta:eq}. Furthermore, 
integrating by parts $m$ times 
we get the following formula for the kernel:
\begin{equation*}
\frac{1}{2\pi} \int e^{i\xi(x-y)} g^{(m)}(x, y, \xi)d\xi,\ 
g^{(m)}(x, y, \xi) = i^m \frac{\p^m a(\xi)}{(x-y)^m}.
\end{equation*}
By definition of $h_1, h_2$ we obtain
\begin{equation*}
P_{n, 0}(x, y, \xi; g^{(m)})\le 
C\frac{|\p^m a(\xi)|}{|x|^m + |y|^m + r^m}.
\end{equation*}
Since $m\ge n = [q^{-1}]+1$, the right-hand side 
belongs to $\plainl{q}(\plainL1)(\R^3)$, and the quasi-norm 
is bounded from above by $\1 \p^m a\1_{1, q}$. Now the 
estimate \eqref{symbol_sep_onedim:eq} follows from Theorem 
\ref{full:thm}. 
\end{proof}

\subsection{Trace-class estimates} For $q = 1$ the lattice quasi-norms 
in Theorems \ref{full:thm} and \ref{symbol:thm} coincide with the 
standard $\plainL1$-norms. Due to the relative simplicity 
of these bounds 
it seems appropriate to write them out separately. 
Moreover making the change $\a\bxi = \bxi'$ we can immediately 
extend them to all values $\a\ge 1$:
 
\begin{thm}\label{trace_class_ampl:thm}
Suppose that $\BT\in GL(2, \R)$ satisfies \eqref{T:eq}. Let 
$h_1, h_2~\in~\plainl{\infty}(\plainL{2})(\R^d)$, 
and $P_{d+1,m}\in \plainL1(\R^{3d})$, 
with some $m = 0, 1, \dots$. Then for any $\a\ge 1$ we have
\begin{align*}
\| h_1 \op^{\rm a}_\a(p_{\BT}) h_2\|_{\GS_1}
\le &\ C_{m} \a^d \1 h_1\1_{2, \infty}\ \1 h_2\1_{2, \infty}\\[0.2cm]
&\ \quad\sum_{n_1, n_2 = 0}^{d+1}\sum_{l=0}^m\iiint 
\frac{|\nabla_{\bw}^{n_1}
\nabla_{\bz}^{n_2}\nabla_{\bxi}^l p(\bw, \bz, \bxi)|}
{1+|\tau \bw - \bz|^m} d\bw d\bz d \bxi,
\end{align*}
with a constant $C_{m} = C_{m}(t_0, \d_0)$.
\end{thm}

\begin{thm}\label{trace_class_symbol:thm} 
 Let 
$h_1, h_2~\in~\plainl{\infty}(\plainL{2})(\R^d)$,\ and $\a\ge 1$.
 \begin{enumerate}
 \item 
 Suppose that $F_{d+1, d+1}(a)\in \plainL1(\R^{2d})$. 
Then for any $t\in [0, 1]$ we have  
\begin{equation*}
\|h_1 \op_{\a, t}(a) h_2\|_{\GS_1} 
\le  C \a^d \1 h_1\1_{2, \infty}
\1 h_2\1_{2, \infty}
\sum_{k, l=0}^{d+1}\iint  |\nabla_{\bw}^k\nabla_{\bxi}^l a(\bw, \bxi)| d\bw d\bxi.
\end{equation*}
\item
Suppose that 
the distance between the supports of the functions $h_1, h_2$ is 
at least $r\ge 1$. If 
$F^{\circ}_{d+1, m}(a) \in \plainL1(\R^{2d})$, 
$m \ge d+1$, then for any $t\in [0, 1]$ 
we have 
\begin{equation*}
\|h_1 \op_{\a, t}(a) h_2\|_{\GS_1} 
\le  C_{ m} (\a r)^{d - m}\1 h_1\1_{2, \infty}
\1 h_2\1_{2, \infty}
\sum_{k=0}^{d+1}\iint  |\nabla_{\bw}^k\nabla_{\bxi}^m a(\bw, \bxi)| d\bw d\bxi.
\end{equation*}
\end{enumerate}
The constants $C$ and $C_{m}$ do not depend on $t\in [0, 1]$.
\end{thm} 
 
For $\BT = \BI$ and $t=0, 1$ the above estimates were obtained in 
\cite{Sob}.

\section{Amplitudes from classes $\BS^{(n_1, n_2, m)}$:
semi-classical estimates}\label{semi_class:sect}

\subsection{Compactly supported amplitudes/symbols}
Now we proceed to estimates for arbitrary $q\in (0, 1]$ 
for the operators containing the parameter $\a>0$.  
Due to the nature of the bounds derived in the previous section 
we do not expect the semi-classical bounds to look as simple as 
in Theorems \ref{trace_class_ampl:thm} and \ref{trace_class_symbol:thm}.
Thus we do not try to find integral bounds but instead we concentrate 
on the scaling properties of the $\GS_q$-estimates. 
For arbitrary numbers $\ell>0$ and $\rho >0$ introduce the following norms:
\begin{equation}\label{norm:eq}
\SN^{(n_1, n_2, m)}(p; \ell, \rho)
= \underset{\substack
{0\le n\le n_1\\
0\le k\le n_2\\
0\le r\le m}}
\max \ \underset{\bw, \bz, \bxi}
\sup \ell^{n+k} \rho^{r}
|\nabla_{\bw}^{n}\nabla_{\bz}^k\nabla_{\bxi}^r p(\bw, \bz, \bxi)|.
\end{equation}
We say that $p$ belongs to the class $\BS^{(n_1, n_2, m)}$ if the norm 
\eqref{norm:eq} is finite for some (and hence for all) positive $\ell, \rho$. 
For a symbol $a = a(\bw, \bxi)$ 
(resp. function $a = a(\bxi)$)  
we use the notation $\SN^{(n, m)}(a; \ell, \rho)$ (resp. $\SN^{(m)}(a; \rho)$). Accordingly, we define classes $\BS^{(n, m)}$ and $\BS^{(m)}$.   
The presence of the parameters $\ell$, $\rho$ allows one to consider
amplitudes and symbols 
with different scaling properties. 

Let $U_\ell$ be the unitary operator on $\plainL2(\R^d)$ 
defined by 
\begin{equation*}
(U_\ell u)(\bx) = \ell^{\frac{d}{2}} u(\ell \bx).
\end{equation*}
Then a straightforward calculation gives for any $\ell, \rho >0$ 
the following unitary equivalence:
\begin{equation}\label{unitary:eq}
U_{\ell}\op_\a^{\rm a} (p_\BT) U_{\ell}^{-1} 
= \op^{\rm a}_{\b}(p^{(\ell, \rho)}_\BT), \
p^{(\ell, \rho)}(\bw, \bz, \bxi) = p(\ell\bw, \ell\bz, \rho\bxi),\
\b = \a\ell\rho.
\end{equation}
The norms \eqref{norm:eq} are also invariant:
\begin{equation}\label{scale:eq}
\SN^{(n_1, n_2, m)}(p; \ell, \rho)
= \SN^{(n_1, n_2, m)}(p^{(\ell_1, \rho_1)}; \ell\ell_1^{-1}, \rho\rho_1^{-1}),
\end{equation}
for arbitrary positive $\ell, \ell_1, \rho, \rho_1$. 

The operators $\op_\a^{\rm a}(p_{\BT})$ transform in a standard way under Euclidean 
isometries ( i.e. orthogonal 
transformations and shifts), their norms \eqref{norm:eq} remain 
invariant.  We use these facts regularly without introducing formal notation for 
these transformations. 
 
All the $\GS_q$-bounds below will be derived under
the following conditions on the amplitudes or symbols. 
For the operator $\op_\a^{\rm a}(p_\BT)$ we assume that
\begin{equation}\label{p_support1:eq}
\textup{the support of }\ \ p = p(\bw, \bz, \bxi)\ \
\textup{is contained in} \ \
B(\bu, \ell)\times \R^d\times B(\bmu, \rho),
\end{equation}
with some $\bu, \bmu\in\R^d$ and some $\ell>0, \rho>0$. 
For the  $t$-operators $\op_{\a, t}(a)$ we assume that
\begin{equation}\label{a_support:eq}
\textup{the support of $a = a(\bw, \bxi)$ }\
\textup{is contained in} \ \
B(\bu, \ell)\times B(\bmu, \rho).
\end{equation}
In what follows most of the bounds are obtained under the assumption that 
$\a\ell\rho\ge \ell_0$ with some fixed positive number $\ell_0$. 
The constants featuring in all the estimates below are independent of the 
symbols involved as well as of the parameters  
$\bu, \bmu, \a, \ell, \rho$ but may depend 
on the constant $\ell_0$. 

\begin{thm}\label{general_trace:thm} 
Let $\BT\in GL(2, \R)$ be a matrix satisfying \eqref{T:eq}, and 
let $s, t\in [0, 1]$.  
Let $q\in (0, 1]$  and $\a\ell\rho\ge \ell_0$. 
Let $p\in \BS^{(n, n, n)}$, with $n$ defined in \eqref{n:eq}, 
be an amplitude satisfying the condition \eqref{p_support1:eq}, 
and let $a\in \BS^{(n, n)}$ be
a symbol satisfying the condition \eqref{a_support:eq}. 
Then $\op_\a^{\rm a}(p_\BT)\in\GS_q$, 
$\op_{\a, t}(a)\in\GS_q$, and
\begin{equation}\label{trace_ampl:eq}
\|\op_\a^{\rm a}(p_\BT)\|_{\GS_q}\le C_q (\a\ell\rho)^{\frac{d}{q}}
\SN^{(n, n, n)}(p; \ell, \rho),
\end{equation}
with a constant $C_q = C_q(t_0, \d_0)$ (see  \eqref{t0:eq}), and 
\begin{equation}\label{trace_symbol:eq}
\|\op_{\a, t} (a)\|_{\GS_q} 
\le C_q (\a\ell\rho)^{\frac{d}{q}} \SN^{(n, n)}(a; \ell, \rho),
\end{equation}
with a constant $C_q$ independent of $t$. 
If, in addition $a\in \BS^{(n, n+1)}$ then 
\begin{equation}\label{l-r:eq}
 \| \op_{\a, t}(a) - \op_{\a, s}(a)\|_{\GS_q}
 \le C_q(\a\ell\rho)^{\frac{d}{q}-1} \SN^{(n, n+1)}(a; \ell, \rho),
\end{equation}
with a constant $C_q$ independent of $s, t$.
\end{thm}

\begin{proof}
The estimate \eqref{trace_symbol:eq} is a special case of \eqref{trace_ampl:eq} 
with the matrix $\BT$ defined in \eqref{tquant1:eq}. 

Without loss of generality we may assume that 
$\bu = \bmu = \bzero$.
Furthermore, 
using \eqref{scale:eq} and \eqref{unitary:eq} 
with $\ell_1 = (\a\rho)^{-1}, \rho_1 = \rho$, we see that it suffices
to prove the sought inequalities for $\a = 1$, $\rho = 1$ and arbitrary 
$\ell\ge \ell_0$ with a fixed $\ell_0 >0$. 
 
Proof of \eqref{trace_ampl:eq}. 
We use Theorem  \ref{full:thm} with $h_1 = h_2 = 1$ and $m=n$. 
Assume without loss of generality that $\SN^{(n, n, n)}(p; \ell, 1) = 1$. 
As $\ell\ge \ell_0$, we have 
\begin{equation*}
 P_{n, n}(\bw, \bz, \bxi; p)\le C 
\frac{\chi_{\bzero, \ell}(\bw)\chi_{\bzero, 1}(\bxi)}{1+|\tau \bw - \bz|^n}, 
\end{equation*}
and hence, for any $\bk, \bs, \bj\in\Z^d$ we have 
\begin{equation*}
\int\limits_{\CC_\bk}\int\limits_{\CC_\bs}
\int\limits_{\CC_\bj}
P_{n, n}(\bw, \bz, \bxi; p) 
d\bw d\bz d\bxi
\le C \frac{\chi_{\bzero, R\ell}(\bj)\chi_{\bzero, 2\sqrt d}(\bk)}
{1 + |\tau\bj-\bs|^n}, 
\end{equation*}
where $R = R(\ell_0) = \ell_0^{-1}\sqrt d + 1$. 
As a consequence,
\begin{equation*}
\1 P_{n, m}(p)  \1_{1, q}\le 
C\biggl(\sum_{|\bj|\le R\ell} \sum_{\bs} \frac{1}{1+ |\tau\bj-\bs|^{nq}}
\biggr)^{\frac{1}{q}}\le C\ell^{\frac{d}{q}},\ C = C(\ell_0), 
\end{equation*}
as $n = [dq^{-1}] + 1> dq^{-1}$. This leads to \eqref{trace_ampl:eq}.

Proof of \eqref{l-r:eq}. We use Theorem \ref{full:thm} with 
$h_1 = h_2 = 1$ and $m = n+1$. 
Without loss of generality assume temporarily 
that 
$\SN^{(n, n+1)}(a) = 1$. 
Rewrite the difference on the left-hand side of \eqref{l-r:eq} 
in the form
\begin{equation*}
\op_{\a, t}(a) - \op_{\a, s}(a)
= \op^{\rm a}_{\a}(g_\BS),
\end{equation*}
with $g(\bw, \bz) = a(\bw, \bxi) - a(\bz, \bxi)$ and the matrix
\begin{equation*}
\BS = 
\begin{pmatrix}
1-t & t\\
1-s & s
\end{pmatrix}.
\end{equation*}
Note that $\det\BS = s-t$, and assume that $|s-t|\ge 1/4$. 
For all $n_1, n_2\le n, l\le n+1$ we have 
\begin{align*} 
|\nabla_{\bw}^{n_1} \nabla_{\bz}^{n_2} \nabla_{\bxi}^l 
g(\bw, \bz, \bxi)|
\le &\ \ell^{-n_1 - n_2} 
\bigl(\chi_{\bold0, \ell}(\bw)+ \chi_{\bold0, \ell}(\bz)\bigr) \chi_{\bold0, 1}(\bxi), \\
|\nabla_{\bxi}^l g(\bw, \bz, \bxi)| \le &\  \ell^{-1} |\bw-\bz|  
\bigl(\chi_{\bold0, \ell}(\bw)+ \chi_{\bold0, \ell}(\bz)\bigr) \chi_{\bold0, 1}(\bxi).
\end{align*}
Therefore
\begin{equation*}
P_{n, n+1}(\bw, \bz, \bxi; g)
\le
C\ell^{-1}\frac{\bigl(\chi_{\bold0, \ell}(\bw)
+ \chi_{\bold0, \ell}(\bz)\bigr) \chi_{\bold0, 1}(\bxi)}{1+|\bw-\bz|^n}.
\end{equation*}
Arguing as in the first part of the proof we arrive at the estimate
\begin{equation*}
\1 P_{n, n+1}(g)  \1_{1, q} \le C\ell^{\frac{d}{q}-1},\ C = C(\ell_0), 
\end{equation*}
which implies \eqref{l-r:eq} by virtue of Theorem \ref{full:thm}. 
As we have assumed that $|\det\BS| 
\ge ~1/4$, 
the constant 
in \eqref{l-r:eq} does not depend on $s, t$. 

If $|s-t|< 1/4$, then we choose a number $u\in [0, 1]$ such that $|s-u|\ge 1/4$,\ 
$|t-u|\ge 1/4$, apply the estimate obtained in the first part of the proof 
to $\op_{\a, s}(a) - \op_{\a, u}(a)$ and 
$\op_{\a, t}(a) - \op_{\a, u}(a)$,
and use the triangle inequality \eqref{triangle:eq}. 
\end{proof}

\begin{thm}\label{separate:thm}
Let $q\in (0, 1]$, $\a\ell\rho\ge \ell_0$ and $R\ge \ell$. 
Let $h_1, h_2\in\plainL\infty(\R^d)$ 
be two functions such that the distance between their supports 
is at least $R$. 
Let $a\in \BS^{(n, m)}, m\ge n,$ be
a symbol satisfying the condition \eqref{a_support:eq}. 
Then  for any $t\in [0, 1]$ we have
\begin{equation*}
\|  h_1 \op_{\a, t}(a) h_2\|_{\GS_q} 
\le C_{q, m} \| h_1\|_{\plainL\infty} 
\|h_2\|_{\plainL\infty}(\a R \rho)^{\frac{d}{q} - m} \SN^{(n, m)}(a; \ell, \rho),
\end{equation*} 
with a constant $C_{q, m}$ independent of $t$. 
\end{thm} 

\begin{proof}
Using \eqref{scale:eq} and \eqref{unitary:eq} with 
$\ell_1 = \ell, \rho_1 = (\a\ell)^{-1}$, we see 
that it suffices to prove the sought inequality 
for $\a = 1$, $\ell = 1$, and arbitrary $\rho\ge \ell_0$  and $R\ge 1$. 
Again, without loss of generality assume that $\bu = \bmu = \bold0$, 
$\|h_1\|_{\plainL\infty} = \|h_2\|_{\plainL\infty}=1$,  
and $\SN^{(n, m)}(a; 1, \rho) = 1$.  
Use Theorem \ref{symbol:thm}(2) with $r = R$. 
It is straightforward to see that 
\begin{equation*}
F^{\circ}_{n, m}(\bw, \bxi; a)
\le C \chi_{\bold0, 1}(\bw)\chi_{\bold0, \rho}(\bxi)
\rho^{-m}, 
\end{equation*}
see \eqref{F:eq} for definition, so that 
\begin{equation*}
\1 F^{\circ}_{n, m}(a)\1_{1, q}^q\le C_q\rho^{d - mq}.
\end{equation*}
By \eqref{symbol0:eq},
\begin{equation*}
\| h_1 \op_{1, t} (a) h_2\|_{\GS_q}
\le 
C_q (R\rho)^{\frac{d}{q}-m},
\end{equation*}
which leads to the sought estimate. 
\end{proof}
 
 \subsection{Symbols with non-compact support} 
 Here we illustrate the use of the obtained estimates 
 and derive a semi-classical 
 bound for the $t$-pseudo-differential operators 
 whose symbols are not necessarily compactly supported. 
 Suppose that for some constant $A>0$, and some number $q\in (0, 1]$  
 the symbol $a$ satisfies the bound 
 \begin{equation}\label{non_compact:eq}
 \underset{\substack
{0\le k\le n\\
0\le l\le n}}
 \max 
 |\nabla_{\bw}^k\nabla_{\bxi}^l a(\bw, \bxi)|\le A
 (1+|\bw|)^{-\g_1} (1+|\bxi|)^{-\g_2}, \g_1, \g_2 > dq^{-1},
 \end{equation}
where $n$ is as in \eqref{n:eq}.

\begin{thm} 
Let the symbol $a$ satisfy \eqref{non_compact:eq}. Then 
\begin{equation*}
\| \op_{\a, t}(a)\|_{\GS_q}\le C_q A \a ^{\frac{d}{q}},
\end{equation*}
with a constant $C_q$ independent of $t\in [0, 1]$.
\end{thm}

\begin{proof} 
As in the proof of Theorem \ref{full:thm} cover $\R^d$ with 
open balls $B(\bj, 2\sqrt d)$, $\bj\in\Z^d$. Let $\psi_\bj\in \plainC\infty_0(\R^d)$, 
$\bj\in\Z^d,$ 
be an associated partition of unity satisfying \eqref{uniform:eq}. Consider the symbols 
\begin{equation*}
a^{(\bj, \bs)}(\bw, \bxi) = \psi_\bj(\bw) \psi_\bs(\bxi) a(\bw, \bxi).
\end{equation*}
These symbols are compactly supported and 
\begin{equation*}
\SN^{(n, n)}(a^{(\bj, \bs)}; 1, 1)\le C A (1+|\bj|)^{-\g_1}(1+|\bs|)^{-\g_2}.
\end{equation*}
By \eqref{trace_symbol:eq},
\begin{equation*}
\| \op_{\a, t}(a^{(\bj, \bs)})\|_{\GS_q}^q\le C A^q\a^d 
(1+|\bj|)^{-\g_1 q}(1+|\bs|)^{-\g_2 q}.
\end{equation*} 
By the triangle inequality \eqref{triangle:eq} we have
\begin{equation*}
\| \op_{\a, t}(a)\|_{\GS_q}^q
\le C A^q  \a^d \sum_{\bj, \bs\in\Z^d} (1+|\bj|)^{-\g_1 q}(1+|\bs|)^{-\g_2 q}
\le C' A^q \a^d,
\end{equation*}
as claimed. 
\end{proof}


\section{Estimates for operators with non-smooth symbols}
\label{nonsmooth:sect}

\subsection{Lipschitz domains} 
\label{domains:subsect}
Here we obtain $\GS_q$-estimates for operators 
with symbols having jump discontinuities. The discontinuities are 
introduced via the projections  
$\chi_\L$ and/or $P_{\Om, \a} = \op_\a(\chi_\Om)$
where $\L$ and $\Om$ are some suitable domains whoce properties are specified 
in the next definition.

\begin{defn}\label{domain:defn}
Let $d\ge 2$. We say that $\L\subset\R^d$ is a  
\textit{basic domain} (or \textit{basic Lipschitz domain}) if  
there exists a Lipschitz function 
$\Phi = \Phi(\hat\bx),\ \hat\bx\in\R^{d-1},$ such that 
with a suitable choice of the Cartesian coordinates $\bx = (\hat\bx, x_d)$, 
$\hat\bx = (x_1, x_2, \dots, x_{d-1})$ the domain $\L$ 
is represented as
\begin{equation}\label{Lambda:eq}
 \L = \{\bx\in\R^d: x_d > \Phi(\hat\bx)\}.  
 \end{equation}
 It is assumed that the function $\Phi$ is \textsl{uniformly} Lipschitz, 
 i.e. the 
 constant 
  \begin{equation}\label{Mphi:eq}
 M = M_{\Phi} = \sup_{\substack{\hat\bx, \hat\by,\\
 \hat\bx\not = \hat\by 
} } 
 \frac{|\Phi(\hat\bx) - \Phi(\hat\by)|}
 {|\hat\bx - \hat\by|}
 \end{equation}
 is finite. 
In this case we use the notation $\L = \G(\Phi)$. 

A domain $\L$ is said to be \textit{Lipschitz} if locally it can be represented by 
basic domains, i.e.  
for any $\bz\in\R^d$ there is a radius $r >0$ such that 
$B(\bz, r)\cap \L = B(\bz, r)\cap \L_0$ with some basic domain $\L_0 = \L_0(\bz)$. 
%
\end{defn}

Our results are also applicable in the case $d = 1$. 
To state them simultaneously for all dimensions, 
in the case $d=1$ we use the term \textit{basic domain} for the domain $\L$ which is 
either $(-\infty, 0)$ or $(0, \infty)$. 
The role of  Lipschits domains will be played by intervals of the 
form $(0, L)$, $L >0$.

Our objective is to obtain semi-classical 
$\GS_q$-estimates for the Hankel-type operators 
$\chi_\L\op_{\a, t}(a)(I-\chi_\L)$, 
$P_{\Om, \a}\op_{\a, t}(a) (I-P_{\a, \Om})$ 
and $\chi_\L P_{\a, \Om}(I-\chi_\L)$, 
with suitable domains $\L, \Om$ and suitable symbols $a$. 
We work either with $t = 0$ or 
$t=1$. First we establish the sought estimates for basic  
domains 
$\L$ and $\Om$, and then extend 
the result to the general bounded Lipschitz 
ones using appropriate partitions of unity.  

For $d\ge 2$ all the $\GS_q$-estimates obtained for the basic  
domains are uniform in the Lipschitz 
constants $M_\Phi$ and $M_\Psi$ 
satisfying the condition 
\begin{equation}\label{gradient:eq}
\max(M_{\Phi}, M_{\Psi}) \le M,
\end{equation}
with some constant $M$.  
Needless to say, the choice of the coordinates for which 
$\L$ or $\Om$ have the form 
\eqref{Lambda:eq} does not have to be the same for the domains $\L$ and $\Om$.

As in the previous section we assume as a rule that the symbols 
are compactly supported and satisfy the condition \eqref{a_support:eq}. 
The constants in the obtained estimates will be independent of
the symbols, and of $\bu$, $\bmu$ and $\ell, \rho$ but may depend 
on the constant $\ell_0$ in the bound $\a\ell\rho\ge \ell_0$, and, for $d\ge 2$, 
on $M$.  
As mentioned in the Introduction some estimates 
were obtained in \cite{Sob} 
for the class $\GS_1$. Note also that for 
$d\ge 2$ the results of \cite{Sob} require 
$\plainC1$-smoothness of the domains $\L$, $\Om$ whereas in the current paper 
the Lipschitz property suffices. 

We obtain consecutively estimates of two types. 
First we study the operators  
$$
\chi_\L \op_{\a, t}(a)(I-\chi_\L)\ \  \textup{and}\ \  
P_{\Om, \a} \op_{\a, t}(a)(I- \chi_\L).
$$ 
Since 
these operators contain only one characteristic function we refer to this case 
as the case of discontinuity in one variable. 
Next we look at the operators of the form 
$\chi_\L \op_{\a, t}(a) P_{\Om, \a}(I-\chi_\L)$ 
which is naturally referred to as the 
case of  discontinuity in two variables.  

It is useful to remark on the scaling properties of basic   
domains in $d\ge 2$. 
Applying \eqref{unitary:eq} to the characteristic function 
$\chi_\L$, $\L = \G(\Phi)$, we observe that under scaling $U_\ell$ 
the domain $\L$ transforms into $\G(\tilde\Phi)$, where 
$\tilde\Phi(\hat\bx) = \ell\Phi(\ell^{-1}\hat\bx)$. It is obvious that 
$M_{\tilde\Phi} = M_{\Phi}$. 

Let $\L = \G(\Phi)\subset\R^d, d\ge 2,$ be a basic domain. 
By definition \eqref{Mphi:eq}, 
\begin{equation*}
|x_d - \Phi(\hat\bx) - \bigl(y_d - \Phi(\hat\by)\bigr)|
\le \lu M\ru\ |\bx - \by|, \ \lu M\ru := \sqrt {1+M^2}
\end{equation*}
for all $\bx, \by\in\R^d$, so that 
\begin{equation}\label{away:eq}
|\bx-\by|\ge \frac{1}{\lu M\ru} \bigl(x_d - \Phi(\hat\bx)\bigr),\ 
\ \ \textup{for all}\ \ \bx\in\L, \by\notin\L.
\end{equation}
In the case $d = 1$, for a basic domain $\L$ 
the same type of bound is obvious:
\begin{equation}\label{away1:eq}
|x-y|\ge |x|, x\in \L, y \notin\L.
\end{equation}

\subsection{ Discontinuity in one variable} 
Here we study the combinations involving an operator  with a smooth symbol and one 
of the operators $\chi_\L$ or $P_{\Om, \a}$.

\begin{thm}\label{sandwich:thm} 
Let $\L$ and $\Om$ be basic domains. 
Let $q\in (0, 1]$, $\a\ell\rho\ge \ell_0$, $n$ be as in \eqref{n:eq}, and let
\begin{equation}\label{m:eq}
m = [(d+1)q^{-1}]+1.
\end{equation}
Suppose that the symbol $a\in\BS^{(n, m)}$ 
satisfies \eqref{a_support:eq}. 
Then  for $t = 0$ or $1$ we have
\begin{equation}\label{sandwich1:eq}
\|\chi_{\L} \op_{\a, t}(a) (1-\chi_\L)\|_{\GS_q}\le C_q 
(\a\ell\rho)^{\frac{d-1}{q}} \SN^{(n, m)}(a; \ell, \rho),
\end{equation}
\begin{equation}\label{sandwich_om:eq}
\|P_{\Om, \a} \op_{\a, t}(a) (1- P_{\Om,\a})\|_{\GS_q}\le C_q 
(\a\ell\rho)^{\frac{d-1}{q}} \SN^{(m, n)}(a; \ell, \rho).
\end{equation}
\end{thm}

\begin{proof} 
The bound \eqref{sandwich_om:eq} follows from \eqref{sandwich1:eq} 
upon exchanging the roles of the variables $\bx$ and $\bxi$.  Thus it suffices to prove 
\eqref{sandwich1:eq}. 

Proof of \eqref{sandwich1:eq}. Assume without loss of generality that 
$\SN^{(n, m)}(a; \ell, \rho) = 1$. 
We prove \eqref{sandwich1:eq} for the operator $\op_{\a, 0}(a)$ 
only, the case $ t = 1$ is done in the same way. 
 
Let $d\ge 2$.  
We use the same scaling argument as in the proof of 
Theorem \ref{general_trace:thm}, and the fact that the Lipschitz constant 
of the domain $\L$ 
does not change under scaling, see the remark at the end 
of Subsection \ref{domains:subsect}. 
Thus it suffices to prove \eqref{sandwich1:eq} 
for $\a= \rho = 1$ and arbitrary $\ell\ge \ell_0$ with a $\ell_0>0$. 
Moreover without loss of generality assume that $\bu = \bmu = \bold0$. 

Choose the coordinates in such a way that $\L$ is represented as 
in \eqref{Lambda:eq}. 
Denote
\begin{equation*}
\L_s = \{\bx\in \R^d: x_d > \Phi(\hat\bx) + s\},\ s\in \R.
\end{equation*}
By virtue of \eqref{away:eq}, 
\begin{equation*}
|\bx-\by|\ge \frac{s + |x_d - \Phi(\hat\bx)-s|}{\lu M\ru},\ 
\forall \bx\in \L_s,\ \by\not\in \L,\ s >0.
\end{equation*}
Cover the closure $\overline{\L}$ with open balls of radius $2\sqrt d$ 
centred at the lattice points $\bj\in\Z^d$. 
Let $R = 4\lu M\ru \sqrt d$ and denote
\begin{equation*}
\Sigma = \{\bj\in\Z^d: B(\bj, 2\sqrt d)\cap \L\not=\varnothing\},\ 
\Sigma_0 = \{\bj\in\Z^d: \bj\in\L_{R}\},\ \Sigma_1 = \Sigma\setminus\Sigma_0.
\end{equation*} 
These definitions ensure that
\begin{equation*}
\dist \{B(\bj, 2\sqrt d), \complement\L\}
\ge 2\sqrt d + \frac{|j_d-\Phi(\hat\bj)-R|}{\lu M\ru}, \ \textup{for all}\ \ \bj\in\Sigma_0,
\end{equation*}
where $\complement\L = \R^d\setminus\L$.
Let $\psi_\bj, \bj\in\Sigma,$ be a smooth partition of unity subordinate to the 
introduced covering, such that 
\begin{equation*}
|\nabla_\bx^k\psi_{\bj}(\bx)|\le C_k, k = 0, 1, \dots,
\end{equation*}
uniformly in $\bj\in\Sigma$.  Denote $\L_\bj = \L\cap B(\bj, 2\sqrt d)$, and 
\begin{equation*}
T_{\bj} = \chi_{\L_\bj} \op_{1, 0}(\psi_\bj a) (I-\chi_\L).
\end{equation*}
Since $\SN^{(n, m)}(a; 1, 1)\le C\SN^{(n, m)}(a, \ell, 1)\le C$, 
by Theorem \ref{separate:thm} we obtain
\begin{equation*}
\| T_{\bj}\|_{\GS_q}^q \le C \biggl(2\sqrt d + \frac{|j_d-\Phi(\hat\bj)-R|}{\lu M\ru}\biggr)^{d-mq},\ \bj\in\Sigma_0.
\end{equation*}
By the triangle inequality 
\eqref{triangle:eq}, 
\begin{equation}\label{sobr:eq}
\left\|\sum_{\bj\in\Sigma_0}T_{\bj} \right\|_{\GS_q}^q
\le C\sum_{|\hat\bj|\le C\ell} \sum_{j_d\in\Z} \biggl(2\sqrt d + \frac{|j_d-\Phi(\hat\bj)-R|}{\lu M\ru}\biggr)^{d-mq}
\le C'\ell^{d-1},
\end{equation} 
 where we have used the fact that $qm>d+1$, see \eqref{m:eq}. 
For $\bj\in\Sigma_1$ we use the bound
\begin{equation*}
\|T_{\bj}\|_{\GS_q}\le \| \op_{1, 0}(\psi_{\bj} a)\|\le C,
\end{equation*} 
 which follows from \eqref{trace_symbol:eq}. As $\# \Sigma_1\le C\ell^{d-1}$, 
 $C = C(\ell_0)$, 
 with the help of the triangle inequality we obtain 
 \begin{equation*}
\left\|\sum_{\bj\in\Sigma_1}T_{\bj} \right\|_{\GS_q}^q
\le C\sum_{\bj\in \Sigma_1} 1 \le C'\ell^{d-1}.
 \end{equation*}
 Together with \eqref{sobr:eq} this leads to 
\begin{equation*}
\|\chi_\L \op_{1, 0}(a) (I-\chi_\L)\|_{\GS_q}^q\le C\ell^{d-1}.
\end{equation*} 
As explained earlier this bound implies \eqref{sandwich1:eq}. 

The proof in the case $d = 1$ is a simplified version of 
that for $d \ge 2$. 
In particular, instead of \eqref{away:eq} one uses \eqref{away1:eq}. 
We omit the details. 
\end{proof}

\begin{rem}\label{commutator:rem}
It is immediate to obtain from Theorem \ref{sandwich:thm} 
estimates of the form \eqref{sandwich1:eq} and \eqref{sandwich_om:eq} 
for the commutators $[\op_{\a, t}, \chi_\L]$ and $[\op_{\a, t}(a), P_{\Om, \a}]$. 
Indeed, recall 
that $[A, \Pi] = (I-\Pi)A\Pi - \Pi A(I-\Pi)$ for 
any bounded operator $A$ and any projection $\Pi$, and that 
$\bigl(\op_{\a, t}(a)\bigr)^* = \op_{\a, 1-t}(\overline a)$. Thus for $t=0$ or $1$ 
it follows from \eqref{sandwich1:eq} 
that 
\begin{equation*}
\|[\op_{\a, t}(a), \chi_\L]\|_{\GS_q}\le C_q 
(\a\ell\rho)^{\frac{d-1}{q}} \SN^{(n, m)}(a; \ell, \rho),
\end{equation*}
and the same estimate holds for the commutator with $P_{\Om, \a}$. 
\end{rem}

The corollary below extends Theorem \ref{sandwich:thm} 
to arbitrary bounded Lipschitz domains.  

\begin{cor}\label{sandwich_gen:cor} 
Let $\L$ and $\Om$ be bounded Lipschitz domains domains (for $d\ge 2$) 
or open bounded intervals (for $d=1$). 
Let $q\in (0, 1]$, $\a\ell\rho\ge \ell_0$, $n$, $m$  be as in \eqref{n:eq} and 
\eqref{m:eq} respectively. 
Suppose that the symbol $a\in\BS^{(n, m)}$ 
satisfies \eqref{a_support:eq}. 
Then  for $t = 0$ or $1$ we have
\begin{equation}\label{sandwich1_gen:eq}
\|\chi_{\L} \op_{\a, t}(a) (1-\chi_\L)\|_{\GS_q}\le C_q 
(\a\ell\rho)^{\frac{d-1}{q}} \SN^{(n, m)}(a; \ell, \rho),
\end{equation}
\begin{equation}\label{sandwich_om_gen:eq}
\|P_{\Om, \a} \op_{\a, t}(a) (1- P_{\Om,\a})\|_{\GS_q}\le C_q 
(\a\ell\rho)^{\frac{d-1}{q}} \SN^{(m, n)}(a; \ell, \rho).
\end{equation}
The constant $C_q$ in the above estimates may depend on the domains $\L, \Om$.
\end{cor}

\begin{proof} In the proof there is no difference between 
the cases $d = 1$ and $d\ge 2$. 
As in Theorem \ref{sandwich:thm} the bound \eqref{sandwich_om_gen:eq} 
follows from \eqref{sandwich1_gen:eq}. 
 Cover $\overline\L$ with 
finitely open balls $B(\bz_j, r)$, 
$j = 1, 2, \dots, J$ 
where $r$ is chosen in such a way that for each $j$ we have  
$B(\bz_j, 4r)\cap \L ~= ~B(\bz_j, 4r)\cap \L_0$ with some basic  
domain $\L_0 = \L_0(j)$.  
Let $\{\phi_j\}, j = 1, 2, \dots, J,$ 
be a finite partition of unity subordinate to the above covering. 
Due to the triangle inequality \eqref{triangle:eq} it suffices to obtain 
the bound \eqref{pred:eq} for the operators of the form 
\begin{equation*}
T_\a = \chi_{\L}\op_{\a, t}(b)(1-\chi_\L),
\end{equation*}
where $b(\bw, \bxi) = \phi(\bw) a(\bw, \bxi)$, and $\phi$ is an element 
of the partition above supported in the ball 
$B(\bz, r)$. Here we have omitted the index $j$ for brevity. 
  If $\L$  had 
been a basic domain then the required bound would have followed 
from \eqref{sandwich2:eq}. Let $\L_0$ be a basic domain 
such that 
\begin{equation}\label{local1:eq}
B(\bz, 4r)\cap \L = B(\bz, 4r)\cap \L_0.
\end{equation}
By construction,
\begin{equation*}
T_\a  = \chi_{\L_0}\op_{\a, t}(b)(I-\chi_\L). 
\end{equation*}
Now we need to show that the estimate \eqref{sandwich1_gen:eq} 
is preserved if 
one replaces $\L$ with $\L_0$ in the last bracket on the right-hand side. 
Let $\z\in\plainC\infty(\R^d)$ be as defined in \eqref{zeta:eq}, and let
$h(\bx) = \z\bigl((|\bx-\bz|(4r)^{-1})\bigr)$, $\tilde h = 1- h$. 
Observe that the distance between the supports of $\phi$ and $h$ is at least $r$. 
Thus by Theorem 
\ref{separate:thm} we have 
\begin{align*}
\|\chi_{\L_0} \op_{\a, t}(b)(I-\chi_\L)\|_{\GS_q}^q
\le &\ \|\op_{\a, t}(b) h\|_{\GS_q}^q
+\ \|\chi_{\L_0} \op_{\a, t}(b)\tilde h(I-\chi_\L)\|_{\GS_q}^q\\[0.2cm]
\le &\ C_m(\a r)^{d - mq} 
+ \|\chi_{\L_0} \op_{\a, t}(b)\tilde h(I-\chi_{\L_0})\|_{\GS_q}^q.
\end{align*}
Here we have used \eqref{local1:eq}. 
In a similar way we show that the last term on the right-hand side  
is bounded by 
\begin{equation*}
C_m(\a r)^{d - mq} 
+ \|\chi_{\L_0} \op_{\a, t}(b)(I-\chi_{\L_0})\|_{\GS_q}^q.
\end{equation*}
Since $\L_0$ is a basic domain we can use \eqref{sandwich1:eq} 
to obtain \eqref{sandwich1_gen:eq} for the symbol $b$. 
As explained earlier, this leads to \eqref{sandwich1_gen:eq} for 
the symbol $a$.  
\end{proof}

\subsection{Discontinuity in two variables} 
In this subsection we prove 
analogues of Theorem \ref{sandwich:thm} and Corollary 
\ref{sandwich_gen:cor} 
with the smooth symbol $a$ replaced   
by the symbol $a(\bx, \bxi) \chi_\Om(\bxi)$.  
Now we need a partition of unity of a special type which 
is described in 
\cite{H}, Ch.~1. 

\begin{prop}\label{hor:prop}
 Let $\tau = \tau(\bxi)>0$ be a Lipschitz function on $\R^d$ such that
\begin{equation}\label{slow:eq}
|\tau(\bxi) - \tau(\boldeta)|\le \vark |\bxi-\boldeta|,\
\end{equation}
for all $\bxi, \boldeta\in\R^d$ with some $\vark\in [0, 1)$. Then there exists
a set $\bxi_j\in\R^d$, $j\in\mathbb N$ such that the balls
$B(\bxi_j, \tau(\bxi_j))$ form a covering of $\R^d$ with the finite intersection property,
i.e. each ball intersects no more than $N = N(\vark)<\infty$ other balls.
Furthermore, there exist non-negative functions $\psi_j\in\plainC\infty_0(\R^d)$,
$j\in\mathbb N$, supported in $B(\bxi_j, \tau(\bxi_j))$ such that
\begin{equation*}
\sum_j \psi_j(\bxi) = 1,
\end{equation*}
and
\begin{equation*}
|\nabla^m\psi_j(\bxi)|\le C_m \tau(\bxi)^{-m},
\end{equation*}
for all $m$ uniformly in $j$.
\end{prop}
 
Assume that $\L, \Om\subset\R^d$ are basic domains. 
For $d\ge 2$ we choose the coordinates in such  way that  
\begin{equation*}
\Om = \{\bxi = (\hat\bxi, \xi_d)\in\R^d: \xi_d > \Psi(\hat\bxi)\}, 
\end{equation*} 
with a Lipschitz function $\Psi$. 
For our purposes the convenient choice of $\tau(\bxi)$ for all 
$\bxi\in\R^d$ is
\begin{equation}\label{dist_rho:eq}
\tau(\bxi) = 
\dfrac{1}{32\lu M\ru}\bigl((\xi_d-\Psi(\hat\bxi))_+^2 
+ \a^{-2}\bigr)^{\frac{1}{2}},
\end{equation}
with the number $M$ as in \eqref{gradient:eq}. 
Since $|\nabla \tau|\le 1/16$,
the condition \eqref{slow:eq} is satisfied
with $\vark=1/16$. 

In the case $d=1$ we  let 
\begin{equation}\label{dist1:eq}
\tau(\xi) = \frac{1}{32}\bigl(|\xi|^2+\a^{-2}\bigr)^{\frac{1}{2}}.
\end{equation}

\begin{thm} 
Let $\L$ and $\Om$ be basic domains. 
Let $q\in (0, 1]$, n be as in \eqref{n:eq}, and 
let $m$ be as in \eqref{m:eq}.  
Suppose that the symbol $a\in\BS^{(n, m)}$ 
satisfies \eqref{a_support:eq}. 
Assume that $\a\ell\rho\ge 2$. Then  
for $t=0$ or $1$ we have 
\begin{equation}\label{sandwich2:eq}
\|\chi_{\L}\op_{\a, t}(a)P_{\Om, \a} (1-\chi_\L)\|_{\GS_q}\le C_q
\bigl((\a\ell\rho)^{d-1} \log (\a\ell\rho)\bigr)^{\frac{1}{q}} 
 \SN^{(n, m)}(a; \ell, \rho).
\end{equation}
\end{thm}

\begin{proof} Suppose that $d\ge 2$. 
Without loss of generality suppose that 
$\SN^{(n, m)}(a; \ell, \rho) = 1$ and $\bmu = \bold0$. 
It suffices to prove the formula \eqref{sandwich2:eq} for $\ell=\rho = 1$ 
and arbitrary $\a\ge 2$. Denote 
\begin{equation*}
T_\a = \chi_{\L}\op_{\a, t}(a)P_{\Om, \a} (1-\chi_\L). 
\end{equation*}
Let $\psi_j, j = 1, 2, \dots$, be a partition of unity 
associated with the function 
\eqref{dist_rho:eq}. 
Let $\tau_j = \tau(\bxi_j)$ be the radii defined in  
Proposition \ref{hor:prop}. 
Then 
\begin{equation}\label{tj:eq}
T_\a = \sum_{j} T^{(j)}_\a,\ 
T^{(j)}_\a = \chi_{\L}\op_{\a, t}(a\psi_j)P_{\Om, \a} (1-\chi_\L).
\end{equation}
Note that $\SN^{(n, m)}(a\psi_j; 1, \tau_j)\le C$ and 
$\a \tau_j\ge (32\lu M\ru)^{-1}$ uniformly in $j$. 
We split the the set of indices $j$ in the sum \eqref{tj:eq} 
into two disjoint parts:
\begin{equation*}
\begin{split}
\Sigma_0 = &\ \{j\in\mathbb N: \supp\psi_j\cap\p\Om \cap
B(\bold0, 1)\not=\varnothing\},\\[0.2cm] 
\Sigma_1 = &\ \{j\in\mathbb N: \chi_\Om\psi_j = \psi_j,\ 
\supp\psi_j \cap B(\bold0, 1)\not = \varnothing\}.
\end{split}
\end{equation*} 
First assume that $j\in \Sigma_0$. By \eqref{dist_rho:eq} 
we have $c\a^{-1}\le \tau_j \le C\a^{-1}$ with some constants $c, C$. 
Thus by \eqref{trace_symbol:eq}, 
\begin{equation*}
\| T^{(j)}_\a\|_{\GS_q}\le 
\|\op_{\a, t}(a\psi_j)\|_{\GS_q}\le C (\a \tau_j)^{\frac{d}{q}}\le \tilde C,
\end{equation*}
uniformly in $j$. 
Since the boundary $\p\Om$ is Lipschitz, 
it is clear that $\#\Sigma_0\le C\a^{d-1}$, and hence by triangle 
inequality \eqref{triangle:eq}, 
\begin{equation}\label{sigma0:eq}
\biggl\|\sum_{j\in\Sigma_0} T^{(j)}_\a\biggr\|_{\GS_q}^q
\le \sum_{j\in\Sigma_0} \bigl\| T^{(j)}_\a \bigl\|_{\GS_q}^q\le C\a^{d-1}. 
\end{equation}
Let us turn to the remaining indices, i.e. to $j\in \Sigma_1$. 
By definition of $\Sigma_1$ we have  
$T^{(j)}_\a = ~\chi_{\L}\op_{\a, t}(a\psi_j)(I-\chi_\L)$, 
$j\in\Sigma_1$, and hence by Theorem \ref{sandwich:thm},
\begin{equation*}
\| T^{(j)}_\a\|_{\GS_q}\le C (\a \tau_j)^{\frac{d-1}{q}}, \ j\in\Sigma_1.
\end{equation*}
Let us sum up all the contributions using the 
triangle inequality \eqref{triangle:eq}:
\begin{align}
\biggl\|\sum_{j\in \Sigma_1}T^{(j)}_\a\biggr\|_{\GS_q}^q
\le &\ \sum_{j\in\Sigma_1} 
\| T^{(j)}_\a\|_{\GS_q}^q
\le C_q \a^{d-1} \sum_{j: |\bxi_j|< 2} \tau_j^{d-1}\notag\\[0.2cm]
\le &\ \tilde C_q \a^{d-1} \underset{\bxi\in\Om,\ 
|\bxi|\le 2}\int \tau(\bxi)^{-1} d\bxi. \label{log:eq}
\end{align}
Here we have used the finite intersection property stated 
in Proposition \ref{hor:prop} and the bounds
\begin{equation*}
(1+\vark)^{-1} \tau(\bxi) \le \tau(\bxi_j)\le (1-\vark)^{-1} \tau(\bxi),\ 
\bxi\in  B\bigl(\bxi_j, \tau(\bxi_j)\bigr).
\end{equation*} 
The integral on the right-hand side of \eqref{log:eq} 
does not exceed 
\begin{equation*}
C\underset{|\hat\bxi|\le 2}\int
\ \ \underset{\substack{\xi_d > \Psi(\hat\bxi),\\ 
|\xi_d|\le 2}}\int 
\frac{1}{\sqrt{\a^{-2} + (\xi_d-\Psi(\hat\bxi))^2}}d\xi_d d\hat\bxi
\le 
C'\int_0^4 \frac{1}{\sqrt{t^2+\a^{-2}}} dt\le C''\log(\a + 1).
\end{equation*} 
Together with \eqref{sigma0:eq} this leads to 
\begin{equation*}
\|T_\a\|_{\GS_q}^q\le C\a^{d-1}\log\a,
\end{equation*}
which implies \eqref{sandwich2:eq}. 

For $d = 1$ the proof follows the same line argument 
and is somewhat simpler. 
We omit the details. 
\end{proof}

Just as before, using an appropriate partition of unity one can deduce 
the following.

\begin{cor}\label{sandwich2:cor}
Let $\L$ and $\Om$ be bounded Lipschitz domains domains (for $d\ge 2$) 
or open bounded intervals (for $d=1$). Let $q\in (0, 1]$. 
Then for any $\a \ge  2$, 
\begin{equation}\label{pred:eq}
\|\chi_{\L}P_{\Om, \a} (1-\chi_\L)\|_{\GS_q}\le C_q
\bigl(\a^{d-1} \log \a\bigr)^{\frac{1}{q}}.
\end{equation}
The constant $C_q$ may depend on the domains $\L, \Om$.
\end{cor}

\begin{proof} 
The proof is similar to that 
of Corollary \ref{sandwich_gen:cor}. 
Cover $\overline\L$ with 
finitely open balls $B(\bz_j, r)$, 
$j = 1, 2, \dots, J$ 
where $r$ is chosen in such a way that  for each $j$, 
$B(\bz_j, 4r)\cap \L = B(\bz_j, 4r)\cap \L_0$ with some basic 
  domain $\L_0 = \L_0(j)$.  
Let $\{B(\bmu_k, r)\}, k = 1, 2, \dots, K$ be a covering of $\overline\Om$ 
with the same properties.   
Let $\{\phi_k\}$ and  $\{\psi_j\}$ 
be finite partitions of unity subordinate to the above coverings. 
Due to the triangle inequality \eqref{triangle:eq} it suffices to obtain 
the bound \eqref{pred:eq} for the operators of the form 
\begin{equation*}
T_\a = \chi_{\L}\op_{\a, 0}(b)P_{\Om, \a} (1-\chi_\L),
\end{equation*}
where $b(\bx, \bxi) = \phi(\bx) \psi(\bxi)$, and $\phi$, $\psi$ 
are elements of the partitions above supported in the balls 
$B(\bz, r)$ and $B(\bmu, r)$. We omit the indices $j, k$ for brevity. 
  If $\L$ and $\Om$ had 
been basic domains then the required bound would have followed 
from \eqref{sandwich2:eq}. Let $\L_0$ and $\Om_0$ be basic 
domains such that 
\begin{equation}\label{local:eq}
B(\bz, 4r)\cap \L = B(\bz, 4r)\cap \L_0, \  
B(\bmu, 4r)\cap \Om = B(\bmu, 4r)\cap \Om_0.
\end{equation}
By construction,
\begin{equation*}
T_\a  = \chi_{\L_0}\op_{\a, 0}(b)P_{\Om_0, \a} (I-\chi_\L). 
\end{equation*}
Now we show that the estimate \eqref{pred:eq} is preserved if 
one replaces $\L$ with $\L_0$ in the last bracket. 
By \eqref{sandwich_om:eq}, 
\begin{align}\label{perv:eq}
\|T_\a\|_{\GS_q}^q\le &\ \|[P_{\Om_0, \a}, \op_{\a, 0}(b)]\|_{\GS_q}^q
+ \|\chi_{\L_0}P_{\Om_0, \a} \op_{\a, 0}(b)(I-\chi_\L)\|_{\GS_q}^q\notag\\[0.2cm]
\le &\ C\a^{d-1} + \|\chi_{\L_0}P_{\Om_0, \a} \op_{\a, 0}(b)(I-\chi_\L)\|_{\GS_q}^q.
\end{align}
In order to estimate the last term on the right-hand side let 
$\z\in\plainC\infty(\R^d)$ be as defined in \eqref{zeta:eq}, and let
$h(\bx) = \z\bigl((|\bx-\bz|(4r)^{-1})\bigr)$, $\tilde h = 1- h$. 
Observe that the distance between the supports of $\phi$ and $h$ is at least $r$. 
Thus by Theorem 
\ref{separate:thm}, for any   $m\ge [dq^{-1}] + 1$ we have 
\begin{align*}
\|\chi_{\L_0}P_{\Om_0, \a} \op_{\a, 0}(b)(I-\chi_\L)\|_{\GS_q}^q
\le &\ \|\op_{\a, 0}(b) h\|_{\GS_q}^q
+\ \|\chi_{\L_0}P_{\Om_0, \a} \op_{\a, 0}(b)\tilde h(I-\chi_\L)\|_{\GS_q}^q\\[0.2cm]
\le &\ C_m(\a r)^{d - mq} 
+ \|\chi_{\L_0}P_{\Om_0, \a} \op_{\a, 0}(b)\tilde h(I-\chi_{\L_0})\|_{\GS_q}^q.
\end{align*}
Here we have used \eqref{local:eq}. 
Reversing the argument for the last term on the right-hand side 
we arrive at the bound 
\begin{equation*}
\|T_\a\|_{\GS_q}^q\le 
C \a^{d-1} + \|\chi_{\L_0} \op_{\a, 0}(b)P_{\Om_0, \a} (I-\chi_{\L_0})\|_{\GS_q}^q.
\end{equation*}
Both domains $\L_0, \Om_0$ are basic, and hence we can use \eqref{sandwich2:eq} 
for the the right-hand side. 
As explained earlier, this leads to \eqref{pred:eq}.  
\end{proof}
   

\bibliographystyle{amsplain}

\end{document}